\documentclass{amsart}
\usepackage{amssymb,amsthm,amsmath,epsfig,latexsym,xypic}
\usepackage{calc}
\usepackage{latexsym}
\usepackage{amscd,amssymb,subfigure} 
\usepackage[arrow,matrix,graph,frame,poly,arc,tips]{xy}

\begin{document}

\newcommand{\mmbox}[1]{\mbox{${#1}$}}
\newcommand{\proj}[1]{\mmbox{{\mathbb P}^{#1}}}
\newcommand{\affine}[1]{\mmbox{{\mathbb A}^{#1}}}
\newcommand{\Ann}[1]{\mmbox{{\rm Ann}({#1})}}
\newcommand{\caps}[3]{\mmbox{{#1}_{#2} \cap \ldots \cap {#1}_{#3}}}
\newcommand{\N}{{\mathbb N}}
\newcommand{\Z}{{\mathbb Z}}
\newcommand{\R}{{\mathbb R}}
\newcommand{\A}{{\mathcal A}}
\newcommand{\C}{{\mathbb C}}
\newcommand{\Tor}{\mathop{\rm Tor}\nolimits}
\newcommand{\ot}{\mathop{\rm OT}\nolimits}
\newcommand{\ao}{\mathop{\rm AOT}\nolimits}
\newcommand{\Ext}{\mathop{\rm Ext}\nolimits}
\newcommand{\Hom}{\mathop{\rm Hom}\nolimits}
\newcommand{\im}{\mathop{\rm Im}\nolimits}
\newcommand{\rank}{\mathop{\rm rank}\nolimits}
\newcommand{\codim}{\mathop{\rm codim}\nolimits}
\newcommand{\supp}{\mathop{\rm supp}\nolimits}
\newcommand{\CB}{Cayley-Bacharach}
\newcommand{\HF}{\mathrm{HF}}
\newcommand{\HP}{\mathrm{HP}}
\newcommand{\coker}{\mathop{\rm coker}\nolimits}
\sloppy
\newtheorem{defn0}{Definition}[section]
\newtheorem{prop0}[defn0]{Proposition}
\newtheorem{conj0}[defn0]{Conjecture}
\newtheorem{thm0}[defn0]{Theorem}
\newtheorem{lem0}[defn0]{Lemma}
\newtheorem{corollary0}[defn0]{Corollary}
\newtheorem{example0}[defn0]{Example}
\newtheorem{rmk0}[defn0]{Remark}

\newenvironment{rmk}{\begin{rmk0}}{\end{rmk0}}
\newenvironment{defn}{\begin{defn0}}{\end{defn0}}
\newenvironment{prop}{\begin{prop0}}{\end{prop0}}
\newenvironment{conj}{\begin{conj0}}{\end{conj0}}
\newenvironment{thm}{\begin{thm0}}{\end{thm0}}
\newenvironment{lem}{\begin{lem0}}{\end{lem0}}
\newenvironment{cor}{\begin{corollary0}}{\end{corollary0}}
\newenvironment{exm}{\begin{example0}\rm}{\end{example0}}

\newcommand{\msp}{\renewcommand{\arraystretch}{.5}}
\newcommand{\rsp}{\renewcommand{\arraystretch}{1}}

\newenvironment{lmatrix}{\renewcommand{\arraystretch}{.5}\small
 \begin{pmatrix}} {\end{pmatrix}\renewcommand{\arraystretch}{1}}
\newenvironment{llmatrix}{\renewcommand{\arraystretch}{.5}\scriptsize
 \begin{pmatrix}} {\end{pmatrix}\renewcommand{\arraystretch}{1}}
\newenvironment{larray}{\renewcommand{\arraystretch}{.5}\begin{array}}
 {\end{array}\renewcommand{\arraystretch}{1}}

\def \a{{\mathrel{\smash-}}{\mathrel{\mkern-8mu}}
{\mathrel{\smash-}}{\mathrel{\mkern-8mu}} {\mathrel{\smash-}}{\mathrel{\mkern-8mu}}}

\title {The Orlik-Terao algebra and 2-formality}

\author{Hal Schenck}
\thanks{Schenck supported by NSF 07--07667, NSA 904-03-1-0006}
\address{Schenck: Mathematics Department \\ University of
 Illinois \\
   Urbana \\ IL 61801\\USA}
\email{schenck@math.uiuc.edu}

\author{\c Stefan O. Toh\v aneanu}
\address{Tohaneanu: Math Department \\ University of Cincinnati \\
  Cincinnati \\ OH 45221\\USA}
\email{ stefan.tohaneanu@uc.edu}

\subjclass[2000]{52C35, 13D02, 13D40} \keywords{Hyperplane arrangement, Free resolution, Orlik-Terao algebra}

\begin{abstract}
\noindent The Orlik-Solomon algebra is the cohomology ring of the complement of a hyperplane arrangement $\A \subseteq \C^n$;
it is the quotient of an exterior algebra $\Lambda(V)$ on $|\A|$ generators. In \cite{ot1}, Orlik and Terao introduced a
commutative analog $Sym(V^*)/I$ of the Orlik-Solomon algebra to answer a question of Aomoto and
showed the Hilbert series depends only on the intersection lattice $L(\A)$. In \cite{fr}, Falk and Randell define the
property of 2-formality; in this note we study the relation between 2-formality and the Orlik-Terao algebra. Our
main result is a necessary and sufficient condition for 2-formality in terms of the quadratic component 
$I_2$ of the Orlik-Terao ideal $I$; 2-formality is determined by the tangent 
space $T_p(V(I_2) \cap (\C^*)^{d-1})$ at a generic point $p$.
\end{abstract}
\maketitle

\section{Introduction}\label{sec:one}
Let $\A=\{H_1,\dots ,H_d\}$ be an arrangement of complex hyperplanes in $\C^n$. 
In \cite{OS}, Orlik and Solomon showed that the cohomology ring of the 
complement $X=\C^n\setminus \bigcup_{i=1}^d H_i$ is determined by the intersection lattice
\[
L(\A)=\{\bigcap_{H\in \A'}\, H \mid \A'\subseteq \A\}.
\]
The Orlik-Solomon algebra $H^*(X,\Z)$ is the quotient of the exterior algebra $E=\bigwedge (\Z^d)$ on generators
$e_1, \dots , e_d$ in degree $1$ by the ideal generated by all elements of the form
\[
\partial e_{i_1\dots i_r}:=\sum_{q}(-1)^{q-1}e_{i_1} \cdots
\widehat{e_{i_q}}\cdots e_{i_r},
\]
for which $\codim H_{i_1}\cap \cdots \cap H_{i_r} < r$.
Throughout this paper, we work with an essential, central arrangement of $d$ hyperplanes; this means we may
always assume $L(\A)$ has rank $n$ and
\[
{\mathcal A} =  \bigcup\limits_{i=1}^d V(\alpha_i) \subseteq {\mathbb
  P}^{n-1},
\]
where $\alpha_i$ are distinct homogeneous linear forms such that $H_i = V(\alpha_i)$. Write $[d]$ for
$\{1,\ldots ,d\}$ and let $\Lambda =\{i_1,\ldots ,i_k\} \subset [d]$. If $codim(\bigcap_{j=1}^k H_{i_j}) < k$,
then there are constants $c_{t}$ such that
\[
\sum\limits_{t \in \Lambda} c_t \alpha_t =0.
\]
In \cite{ot1}, Orlik and Terao introduced the following commutative analog of the Orlik-Solomon algebra in order
to answer a question of Aomoto.
\begin{defn}
For each dependency $\Lambda=\{i_1,\ldots ,i_k\}$, let $r_\Lambda = \sum_{j=1}^k c_{i_j}y_{i_j} \in R
=\mathbb{K}[y_1,\ldots,y_d]$. Define $f_\Lambda = \partial(r_\Lambda) = \sum_{j=1}^k c_{i_j} (y_{i_1}\cdots \hat
y_{i_{j}} \cdots y_{i_k})$, and let $I$ be the ideal generated by the $f_{\Lambda}$. The \textbf{Orlik-Terao
algebra} $\ot$ is the quotient of $\mathbb{K}[y_1,\ldots,y_d]$ by $I$. The \textbf{Artinian Orlik-Terao algebra}
$\ao$ is the quotient of $\ot$ by the squares of the variables.
\end{defn}

Orlik and Terao actually study the Artinian version, but for our purposes the $\ot$ algebra will turn out to be
more interesting. The crucial difference between the Orlik-Solomon algebra and Orlik-Terao algebra(s) is not the
difference between the exterior algebra and symmetric algebra, but rather the fact that the Orlik-Terao algebra
actually records the ``weights'' of the dependencies among the hyperplanes. So in any investigation where actual
dependencies come into play, the $\ot$ algebra is the natural candidate for study.

\subsection{2-Formality}
In \cite{fr}, Falk and Randell introduced the concept of 2-formality:
\begin{defn}\label{defn:Relsp}
For an arrangement ${\mathcal A}$, the {\em relation space}
$F(\mathcal A)$ is the kernel of the evaluation map $\phi$:
\[
\bigoplus\limits_{i=1}^d \mathbb{K}e_i \stackrel{e_i \mapsto \alpha_i}{\longrightarrow}
\mathbb{K}[x_1,\ldots,x_n]_1.
\]
$\A$ is {\em 2-formal} if $F(\mathcal A)$ is spanned by relations involving only three hyperplanes.
\end{defn}
\begin{exm}\label{exm:firstex}
 Suppose we have an arrangement of 4 lines in ${\mathbb P}^2$
given by the linear forms: $\alpha_1=x_1, \alpha_2=x_2, \alpha_3=x_3, \alpha_4=x_1+x_2+x_3$. Obviously any three
of the forms are independent, so the only relation is $\alpha_1+\alpha_2+\alpha_3-\alpha_4 = 0$. Hence the $\ot$
algebra is
\[
\mathbb{K}[y_1,\ldots,y_4]/\langle y_2y_3y_4+y_1y_3y_4+y_1y_2y_4-y_1y_2y_3 \rangle.
\]
This arrangement cannot be 2-formal, since there are no relations involving only three lines, whereas $F(\mathcal
A)$ is nonzero.

\end{exm}Many interesting classes of arrangements are 2-formal: in \cite{fr}, Falk and Randell proved that
 $K(\pi,1)$ arrangements and arrangements with quadratic Orlik-Solomon
ideal are 2-formal. In \cite{bt}, Brandt and Terao generalized the
notion of 2-formality to $k-$formality, proving that every free
 arrangement is $k-$formal. Formality of discriminantal arrangements
is studied in \cite{bb}, with surprising connections to fiber polytopes
\cite{bs}. In \cite{yu}, Yuzvinsky shows that 
free arrangements are 2-formal; and gives an example showing 
that 2-formality does not depend on $L(\A)$.

\begin{exm}\label{Yuz}
Consider the following two arrangements of lines in $\mathbb{P}^2$:
\[
\begin{array}{ccc}
\!\!\!\A_1\!\! &\!\! =\!\! &\! \!V(xyz(x\!+y\!+z)(2x\!+y\!+z)(2x\!+3y\!+z)(2x\!+3y\!+4z)(3x\!+5z)(3x\!+4y\!+5z)) \\
\!\!\!\A_2\!\!& \!\!=\!\!\!\! &\!\!\!\!V(xyz(x\!+y\!+z)(2x\!+y\!+z)(2x\!+3y\!+z)(2x\!+3y\!+4z)(x\!+3z)(x\!+2y\!+3z)).
\end{array}
\]
For a graded $R$--module $M$, the graded betti numbers of $M$ are
\[
b_{ij} = \dim_{\mathbb{K}}\Tor_i^R(M,\mathbb{K})_{i+j}.
\]
We shall be interested in the case when $M=R/I$. The graded betti numbers of the $\ot$ and $\ao$ algebras
of $\A_1$ and $\A_2$ are identical; for $\ot = R/I$ the numbers are:
\begin{small}
$$
\vbox{\offinterlineskip 
\halign{\strut\hfil# \ \vrule\quad&# \ &# \ &# \ &# \ &# \ &# \ &# \ &# \ &# \ &# \ &# \ &# \ &# \  &# \ &# \cr
total&1&26&120&216&190&84&15\cr \noalign {\hrule} 0&1 &--&--& --&--&  --&--&  \cr 1&--&6 &-- & --&--& --&--&
\cr 2&--&20&120 &216&190& 84 & 15 \cr \noalign{\bigskip} \noalign{\smallskip} }}
$$
\end{small}
\noindent This diagram is read as follows: the entry in position $(i,j)$ is simply $b_{ij}$.
So for example,
\[
\dim_{\mathbb{K}}\Tor_2^R(R/I,\mathbb{K})_4 = 120.
\]

\noindent Yuzvinsky shows that $\mathcal A_1$ is 2-formal, and $\mathcal A_2$ is not. The arrangements have the
same intersection lattice and appear identical to the naked eye:

\begin{figure}[h]
\begin{center}
\includegraphics{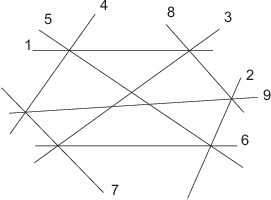}
\caption{\textsf{}}
\label{fig:yuz}
\end{center}
\end{figure}

The difference is that the six multiple points of $\mathcal A_2$ lie on a smooth conic, while the six multiple points
of $\mathcal A_1$ do not. The quadratic $\ot$ algebra is $\mathbb{K}[y_1,\ldots,y_d]/I_2$, where $I_2$ consists
of the quadratic generators of $I$. The graded betti numbers of the  quadratic $\ot$ algebra of $\mathcal A_1$
are:

\begin{small}
$$
\vbox{\offinterlineskip 
\halign{\strut\hfil# \ \vrule\quad&# \ &# \ &# \ &# \ &# \ &# \ &# \ &# \ &# \ &# \ &# \ &# \ &# \ \cr
total&1&6&15&20&15&6&1 \cr \noalign {\hrule} 0&1 &--&--& --&--&--&--  \cr 1&--&6 &-- & --&--&--&--   \cr 2&--&--
&15 & --&--&--&--  \cr 3&--&-- &-- &20&--&--&--  \cr 4&--&--&--&--&15&--&--  \cr 5&--&--&--&--&--&6&--
 \cr 6&--&--&--&--&--&--&1&  \cr \noalign{\bigskip} \noalign{\smallskip} }}
$$
\end{small}
\noindent while the quadratic $\ot$ algebra of $\mathcal A_2$ has betti diagram:

\begin{small}
$$
\vbox{\offinterlineskip 
\halign{\strut\hfil# \ \vrule\quad&# \ &# \ &# \ &# \ &# \ &# \ &# \ &# \ &# \ &# \ &# \ &# \ &# \ \cr
total&1&6&20&31&21&5 \cr \noalign {\hrule} 0&1 &--&--& --&--&--  \cr 1&--&6 &-- & --&--&--   \cr 2&--&-- &20 &
16&5&--  \cr 3&--&-- &-- &15&16&5&  \cr \noalign{\bigskip} \noalign{\smallskip} }}
$$
\end{small}
\end{exm}
This paper is motivated by the question raised by the previous
example: ``Is 2-formality determined by the quadratic $\ot$ algebra?''
Our main result is: \vskip .1in \noindent{\bf Theorem}: Let $\mathcal A$ be an
arrangement of $d$ hyperplanes of rank $n$. $\mathcal A$ is
2-formal if and only if $V(I_2)\cap (\mathbb{C}^*)^{d-1}$ has codimension $d-n$. \pagebreak

The class of 2-formal arrangements is
quite large, and includes free arrangements, $K(\pi,1)$ arrangements,
and arrangements with quadratic Orlik-Solomon ideal. We show below that if
$\codim(I_2) = d-n$, then $\A$ is 2-formal. Le-Mohammadi point out in \cite{LM}
that for the converse, it is necessary to saturate. 
\begin{defn}
Let $G$ be a simple graph on $\nu$ vertices, with edge-set $\mathsf{E}$,
and let $\A_{G}=\{z_i-z_j=0\mid (i,j)\in \mathsf{E} \}$
be the corresponding arrangement in $\C^{\nu}$.
\end{defn}
For a graphic arrangement $\A_G$ it is obvious that $d = |\mathsf{E}|$,
and easy to show that $\rank \A_G = \nu -1$.
In \cite{t}, Tohaneanu showed that a graphic arrangement $\A_{G}$ is
2-formal exactly when $H_1(\Delta_G) = 0$, where $\Delta_G$ is the
clique complex of $G$--a simplicial complex, whose $i$-faces
correspond to induced complete subgraphs on $i+1$ vertices.

\begin{exm}\label{egypt}
The clique complex for $G$ as in Figure~2 consists of the four
triangles (and all vertices and edges). Clearly $H_i(\Delta_G)=0$
for all $i\ge 1$, so $\A_G$ is 2-formal.
\begin{figure}[h]
\setlength{\unitlength}{0.6cm}
\begin{picture}(5,4.8)(0,-2.5)
\xygraph{!{0;<12mm,0mm>:<0mm,12mm>::}
[]*-{\bullet}
(-[dd]*-{\bullet},-[dr]*-{\bullet}
(-[ur],-[dl]),-[rr]*-{\bullet}
(-[dd]*-{\bullet}
(-[ll],-[ul])))
}
\end{picture}
\caption{\textsf{}}
\label{fig:nonhyper}
\end{figure}

\noindent The graded betti numbers of $\ot$ are
\begin{small}
$$
\vbox{\offinterlineskip 
\halign{\strut\hfil# \ \vrule\quad&# \ &# \ &# \ &# \ &# \ &# \ &# \ &# \ &# \ &# \ &# \  \cr
total&1&5&10&9&3 \cr \noalign {\hrule} 0&1 &--&--& --&--  \cr 1&--&4 &-- & --&
--   \cr 2&--&1 &9 & 3&--  \cr 3&--&-- &1 &6&3  \cr 4&--&--&--&--&15  \cr \noalign{\bigskip} \noalign{\smallskip} }}
$$
\end{small}
\noindent while the quadratic $\ot$ algebra of $\mathcal A_G$ has betti diagram:

\begin{small}
$$
\vbox{\offinterlineskip 
\halign{\strut\hfil# \ \vrule\quad&# \ &# \ &# \ &# \ &# \ &# \ &# \ &# \ &# \
 &# \ &# \  \cr
total&1&4&6&4&1 \cr \noalign {\hrule} 0&1 &--&--& --&--  \cr 1&--&4 &-- & --&-
-   \cr 2&--&--
&6 & --&--  \cr 3&--&-- &-- &4&--  \cr 4&--&--&--&--&1
\cr \noalign{\bigskip} \noalign{\smallskip} }}
$$
\end{small}
Since $I_2$ is clearly a complete intersection, $codim(I_2) = 4 = 8-(5-1) = d-
n$, and Corollary~\ref{onedirection} gives another
proof of 2-formality for this configuration.
\end{exm}
While the examples of 2-formal arrangements encountered thus far have $I_2$
a complete intersection, this is generally not the case.
\begin{exm}\label{NF}
The Non-Fano arrangement is the unique configuration of seven lines in $\mathbb{P}^2$ having six triple points. It is free, hence 2-formal. The graded betti numbers of $\ot$ are
\begin{small}
$$
\vbox{\offinterlineskip 
\halign{\strut\hfil# \ \vrule\quad&# \ &# \ &# \ &# \ &# \ &# \ &# \ &# \ &# \ &# \ &# \  \cr
total&1&7&14&12&4 \cr \noalign {\hrule} 0&1 &--&--& --&--  \cr 1&--&6 &5 & --&--   \cr 2&--&1
&9 & 12&4  \cr \noalign{\bigskip} \noalign{\smallskip} }}
$$
\end{small}
\noindent while the quadratic $\ot$ algebra has betti diagram:

\begin{small}
$$
\vbox{\offinterlineskip 
\halign{\strut\hfil# \ \vrule\quad&# \ &# \ &# \ &# \ &# \ &# \ &# \ &# \ &# \ &# \ &# \  \cr
total&1&6&10&6&1 \cr \noalign {\hrule} 0&1 &--&--& --&--  \cr 1&--&6 &5 & --&--   \cr 2&--&--
&5 & 6&--  \cr 3&--&-- &-- &--&1 \cr \noalign{\bigskip} \noalign{\smallskip} }}
$$
\end{small}
So $I_2$ is not a complete intersection, though it is Gorenstein. In
general, $I_2$ need not even be Cohen-Macaulay, and often has
multiple components. For the Non-Fano
arrangement the primary decomposition of $I_2$ is
\[
I_2 = I \cap \langle y_0,y_1,y_3,y_4 \rangle
\]
In particular, if $I \ne I_2$, the primary decomposition is nontrivial.
\end{exm}
When $\A$ is 2-formal, it is possible to give a necessary 
combinatorial criterion for $I_2$ to be a complete intersection.
\begin{cor}\label{CIcrit}
Let $\A$ be a 2-formal arrangement of rank $n$, with $|\A|=d$.
If $I_2$ is a complete intersection, then 
\[
b_2 = {d \choose 2}-d + n
\]
\end{cor}
\begin{proof}
First, since $b_2$ is the dimension of the Orlik-Solomon algebra in degree two,
\[
b_2 = \dim_{\mathbb{K}}\Lambda^2(\mathbb{K}^d) - \dim_{\mathbb{K}}J_2,
\]
where $J$ denotes the Orlik-Solomon ideal. Now, by the result of Orlik-Terao mentioned
earlier, $\ao$ has the same Hilbert function as the Orlik-Solomon algebra. This means
\[
b_2 = \dim_{\mathbb{K}}(R/I+\langle y_1^2,\ldots,y_d^2\rangle)_2
\]
Now, clearly
\[
\dim_{\mathbb{K}}(R/\langle y_1^2,\ldots,y_d^2\rangle)_2 = \dim_{\mathbb{K}}\Lambda^2(\mathbb{K}^d).
\]
In the next section, it is proved that $I$ is prime. It
follows from this that no $y_i^2 \in I_2$.
Combining, we obtain that
\[
\dim_{\mathbb{K}}I_2 = \dim_{\mathbb{K}}J_2.
\]
Theorem~\ref{colonThm} shows $A$ is 2-formal iff $I_2:y_1\cdots y_d = I$; in particular $I$ is an associated prime of $I_2$. If $I_2$ is a complete intersection, it is unmixed, so $\codim(I_2) = d-n$. As $I_2$ is a complete intersection iff $\codim(I_2) = \dim_{\mathbb{K}}(I_2)$, the result follows.
\end{proof}
\section{The quadratic $\ot$ algebra and 2-Formality}\label{sec:two}
We keep the setup of the previous section:
$\A$ is an essential, central arrangement of $d$ hyperplanes in
$\mathbb{P}^{n-1}$, with relation space $F(\A)$.
Since $\dim_{\mathbb{K}}F(\A)=d-n$, $\A$ is
2-formal if and only iff $\dim_K(span_{\mathbb{K}}(\{\mbox{3-relations}\}))=d-n$.
Fix defining linear forms $\alpha_i$ so that $H_i=V(\alpha_i)$,
let $I \subseteq R=\mathbb{K}[y_1,\ldots,y_d]$ denote the OT ideal,
and $I_2$ the quadratic component of $I$.


\begin{prop}\label{Prime}
$I$ is prime, and $V(I)$ is a rational variety of codimension $d-n$.
\end{prop}
\begin{proof}
Consider the map
\[
R \stackrel{\Phi}{\longrightarrow} \mathbb{K}\Bigg[\frac{1}{\alpha_1},\ldots, \frac{1}{\alpha_d}\Bigg] = C(\A)
\]
given by $y_i \mapsto \frac{1}{\alpha_i}$. An easy check shows that $I \subseteq \ker \Phi$.
Our assumption that $\A$ is essential implies that
\[
\mathbb{K}\Bigg[\frac{1}{y_1},\ldots, \frac{1}{y_n}\Bigg] \subseteq  \mathbb{K}\Bigg[\frac{1}{\alpha_1},\ldots, \frac{1}{\alpha_d}\Bigg],
\]
hence the field of fractions of $R/\ker(\Phi)$ is $\mathbb{K}(y_1,\ldots,y_n)$,
giving rationality and the appropriate dimension (as an affine cone).
In \cite{ter}, Terao proved that the Hilbert series for $C(\A)$ is
given by
\[
HS(C(\A),t) = P\Big(\A, \frac{t}{1-t}\Big).
\]
where $P(\A,t)$ is the Poincar\'e polynomial of $\A$.
If $H$ is a hyperplane of $\A$, the {\em deletion} $\A'$
is the subarrangement $\A \setminus H$ of $\A$,
and the {\em restriction} $\A''$ is the arrangement
$\{H' \cap H \mid H'\in \A'\}$, considered as an arrangement
in the vector space $H$, and there is a relation
\[
P(\A,t)=P(\A',t)+tP(\A'',t).
\]
Thus the Hilbert series of $C(\A)$ satisfies
the recursion
\[
HS(C(\A),t) = HS(C(\A'),t)+\frac{t}{1-t}HS(C(\A''),t)
\]
If the quotient $R/I$ satisfies the same
recursion, then $I = \ker \Phi$ will follow by induction. Let $y_1$ be
a variable corresponding to $H = H_1$. In \cite{ps}, Proudfoot and
Speyer prove that the broken circuits are a universal Gr\"obner
basis for $I$, hence in particular a lex basis.
Let $R = \mathbb{K}[y_1,\ldots,y_n]$ and $R'= \mathbb{K}[y_2,\ldots,y_n]$,
and consider the short exact sequence
\[
0 \longrightarrow R(-1)/(in(I):y_1) \stackrel{\cdot y_1}{\longrightarrow}
R/in(I) \longrightarrow R/(in(I),y_1) \longrightarrow 0.
\]
The initial ideal of $I$ has the form
\[
in(I) = \langle f_1,\ldots,f_k, y_1\cdot g_1, \ldots,y_1\cdot g_m \rangle,
\]
with $f_i, g_j$ not divisible by $y_1$. Clearly
\[
in(I):y_1 = \langle f_1,\ldots,f_k, g_1, \ldots, g_m \rangle
\]
In the rightmost module of the short exact sequence, quotienting
by $y_1$ kills the generator and all relations involving $y_1$, hence
\[
R/(in(I),y_1) \simeq R'/\langle f_1,\ldots,f_k \rangle = R'/in(I'),
\]
with $I'$ denoting the $\ot$ ideal of $\A'$. In the leftmost
module, since no relation of $in(I):y_1$ involves $y_1$,
$in(I):y_1 \subseteq R'$, so taking into account the degree
shift and the fact that $y_1$ acts freely on the module, we see
\[
HS(R(-1)/(in(I):y_1),t) = \frac{t}{1-t}HS(R'/ \langle f_1,\ldots,f_k, g_1, \ldots, g_m \rangle,t).
\]
We now claim that
$HS(R'/ \langle f_1,\ldots,f_k, g_1, \ldots, g_m \rangle, t) = HS(C(A''),t)$.
To see this, note that the initial monomials $f_k$ express the fact that the dependencies
among hyperplanes of $\A'$ continue to hold in $\A''$. The
relations $g_j$ represent the ``collapsing'' that occurs on restricting to
$H_1$. For example, if there is a circuit $(123)$ in $\A$, this means
that $H_2|_{H_1} = H_3|_{H_1}$. In $\ot(\A)$, the relation on $(123)$ has
initial term $y_1y_2$, hence $y_2 \in in(I):y_1$. This reflects
the ``redundancy'' $H_2|_{H_1} = H_3|_{H_1}$ in the restriction.
Similar reasoning applies to the relations $g_i$ of degree greater
than one, leading to:
\[
HS(R'/ \langle f_1,\ldots,f_k, g_1, \ldots, g_m \rangle,t) = HS(C(\A''),t)
= P\Big(\A'',\frac{t}{1-t}\Big).
\]
Combining this with an induction and Terao's formula shows that $I = \ker \Phi$.
\end{proof}
\begin{cor}\label{nondegen}
The variety $V(I)$ is nondegenerate.
\end{cor}
\begin{proof}
Since $I$ is prime, $V(I)$ will be contained in a hyperplane
$V(L)$ iff $L \in I$. However, the assumption that the hyperplanes
of $\A$ are distinct implies that there are no dependencies involving
only two hyperplanes, so that $I$ is generated in degree greater than
two. In particular, $I$ contains no linear forms.
\end{proof}

\begin{thm}\label{main1}
Let $\mathcal A$ be an arrangement of $d$ hyperplanes of rank $n$.
$\mathcal A$ is 2-formal if and only if $X = V(I_2)\cap (\mathbb{C}^*)^{d-1}$
has codimension $d-n$.
\end{thm}
\begin{proof}
The proof hinges on the fact that for
$p \in (\mathbb{C}^*)^{d-1} \subseteq \mathbb{P}^{d-1}$,
the Jacobian ideal of $I_2$, evaluated at $p$, is (after
a linear change of coordinates) exactly the
matrix recording the dependencies among triples of hyperplanes.

Let
\[
f=ay_2y_3+by_1y_3+cy_1y_2=\partial(ay_1+by_2+cy_3), a,b,c\neq 0,
\]
be a generator of $I_2$, let $p=(p_1:\ldots:p_d)$ be a
point on $X$, so no $p_i = 0$. Such points
form a dense open subset of $V(I)$ by Corollary~\ref{nondegen}.
The Jacobian matrix of the generators
of $I_2$ has a special form when evaluated at 
$p \in V(I_2) \cap (\mathbb{C}^*)^{d-1}$. 
Write $f_i$ for $\partial(f)/\partial(y_i)$. 
For $f$ as above, if $i \ge 4$ then $f_i = 0$, and 
\begin{eqnarray}
f_{1}&=& by_3+cy_2 \nonumber \\
f_{2}&=& ay_3+cy_1 \nonumber \\
f_{3}&=& ay_2+by_1 \nonumber
\end{eqnarray}
Evaluating the partials of $f$ at $p$ yields:
\[
f_{1}(p)= p_2p_3(b\frac{1}{p_2}+c\frac{1}{p_3}), \mbox{  }
f_{2}(p)= p_1p_3(a\frac{1}{p_1}+c\frac{1}{p_3}), \mbox{  }
f_{3}(p)= p_1p_2(a\frac{1}{p_1}+b\frac{1}{p_2}).
\]
Since $f(p)=0$, $a\frac{1}{p_1}+b\frac{1}{p_2}+c\frac{1}{p_3}=0$, 
and simplifying using this relation, we obtain:
\[
f_{1}(p)= -\frac{p_2p_3}{p_1}a,  \mbox{  }
f_{2}(p)= -\frac{p_1p_3}{p_2}b, \mbox{  }
f_{3}(p)= -\frac{p_1p_2}{p_3}c. \]

\noindent The $p_i$ are nonzero, so rescaling shows that the row of
$J_p(I_2)$ corresponding to $f$ is:
\[
\Big[ \frac{a}{p_1^2}, \frac{b}{p_2^2},\frac{c}{p_3^2}, 0,\ldots, 0 \Big]
\]
Multiplying $J_p(I_2)$ on the right by the diagonal invertible matrix
with $(i,i)$ entry $p_i^2$ yields a matrix whose entries encode the
dependencies among triples of the forms $l_i$. Hence $\rank J_p(I_2)$ is
exactly the dimension of the space of three relations.
\end{proof}
\begin{cor}\label{onedirection}
If  $codim(I_2)=d-n$, then $\A$ is 2-formal.
\end{cor}
\begin{thm}\label{colonThm}
$\A$ is 2-formal if and only if $\codim(I_2:y_1 \cdots y_d)=d-n$.
\end{thm}
\begin{proof}
By Theorem~\ref{main1}, we have
\vskip .05in
$\begin{array}{ccc}
\A \mbox{ is 2-formal }  & \leftrightarrow & \codim(V(I_2) \cap (\C^*)^{d-1})=d-n\\
                         & \leftrightarrow & \codim(V(I_2) \setminus (V(y_1\cdots y_n) \cap V(I_2))=d-n\\
                         & \leftrightarrow & \codim\overline{(V(I_2) \setminus (V(y_1\cdots y_n) \cap V(I_2))}=d-n\\
                         & \leftrightarrow & \codim (I_2:y_1\cdots y_d^\infty)=d-n
\end{array}$
\vskip .05in
To show that a single ideal quotient suffices, let $F(\A)$ be the space of relations 
on the $\alpha_i$; $r \in F(A)$ is of the form $\sum a_iy_i$ with $\sum a_i \alpha_i = 0$. Write $supp(r)$ 
for the indices of $a_i \ne 0$ appearing in $r$. Suppose $r_0,\ldots,r_s \in F(A)$ 
with $\Lambda_i = supp(r_i)$ and $r_0=\sum_{i=1}^s a_ir_i$, and let $\Lambda = \bigcup_{i=1}^n \Lambda_i$.
Then
\[
y_{\Lambda \setminus \Lambda_0} \partial(r) = a_1y_{\Lambda_2\cup\cdots \Lambda_s \setminus \Lambda_1} \partial(r_1)+
\cdots a_sy_{\Lambda_1\cup\cdots \Lambda_{s-1} \setminus \Lambda_s} \partial(r_s).\]
If $\A$ is $2$-formal, then any relation can be written as a sum of $3$-relations, so 
$\partial(r) \in I_2:y_1\cdots y_d$, hence $I \subseteq I_2:y_1\cdots
y_d \subseteq I_2:y_1\cdots y_d^\infty$ and $\codim I_2:y_1\cdots
y_d=d-n = \codim I$ implies $\codim(I_2:y_1\cdots y_d^\infty) = d-n$. In fact, 
\begin{equation}\label{Acriterion}
\A \mbox{ is 2-formal iff }I = I_2:y_1\cdots y_d,
\end{equation}
which follows since $I_2 \subseteq I$ implies $I_2:y_1\cdots y_d
\subseteq I_2:y_1\cdots y_d^\infty \subseteq I:y_1\cdots y_d^\infty \subseteq
I$, since $I$ is prime and nondegenerate. Hence $\codim I_2:y_1\cdots
y_d =d-n = \codim I$ implies $\codim(I_2:y_1\cdots y_d^\infty) = d-n$.
\end{proof}
\begin{rmk}
The published version of Theorem 2.5 omits the ideal quotient. While
Corollary~\ref{onedirection} shows that $\codim(I_2)=d-n$ implies 
$\A$ is $2$-formal, Example 5.1 of \cite{LM} shows that ideal 
quotient is needed for the converse. We thank Le-Mohammadi for 
pointing this out. Equation~\ref{Acriterion} was dropped in 
the published version of our paper.
\end{rmk}
\begin{prop}\label{ss}
If $\A$ is supersolvable, then $I=I_2$.
\end{prop}
\begin{proof} 
Fix the reverse lexicographic order on $R=\mathbb K[y_1,\ldots,y_d]$.
Suppose $C$ is a circuit with $|C|=p\geq 4$, and that 
$\partial(C)\notin I_2$. From the set of circuits with 
$\partial(C)\notin I_2, |C|=p$, select the circuit 
$C=\{H_{j_1},\ldots,H_{j_p}\}, j_1<\cdots <j_p$ which has maximal
lead term $M=y_{j_2}\cdots y_{j_p}$. Since $\A$ is supersolvable,
there exists $j_r,j_s, 1\leq r< s\leq p$ and $u>j_s$ such that 
$D=\{H_{j_r},H_{j_s},H_u\}$ is a circuit.

If $u\in\{j_1,j_2,\ldots,j_p\}$ then $C$ contains $D$, which
would contradict the fact that $C$ is a circuit with $|C|\geq 4$. 
Hence $u\notin\{j_1,j_2,\ldots,j_p\}$. So $D$ gives rise to a dependency: 
\[
D=b_r\alpha_{j_r}+b_s\alpha_{j_s}+\alpha_u=0, b_r,b_s\neq 0,
\]
yielding an element 
\[
\partial(D)=y_{j_r}y_{j_s}+b_ry_{j_s}y_u+b_sy_{j_r}y_u \in I_2.
\]
Note that $C$ also gives a dependency 
$a_1\alpha_{j_1}+\cdots+a_p\alpha_{j_p}=0, a_k\neq 0$ and 
corresponding element 
\[
\partial(C)= a_ry_{j_1}\cdots\widehat{y_{j_r}}\cdots y_{j_p}+a_sy_{j_1}\cdots\widehat{y_{j_s}}\cdots y_{j_p}+ y_{j_r}y_{j_s}P \in I,
\]
with $P=a_1y_{j_2}\cdots\widehat{y_{j_r}}\cdots\widehat{y_{j_s}}\cdots y_{j_p}+\cdots+ a_py_{j_1}\cdots\widehat{y_{j_r}}\cdots\widehat{y_{j_s}}\cdots y_{j_{p-1}}.$
Consider the dependencies $C_r  =a_rD-b_rC$ and $C_s =a_sD-b_sC$. 
Writing $y_{j_r}y_{j_s}=\partial(D)-b_ry_{j_s}y_u-b_sy_{j_r}y_u$ and substituting
into the expression for $\partial(C)$ yields
\[
\partial(C)=\partial(C_r)+\partial(C_s)+\partial(D)P.
\]
Now note that $C_r$ and $C_s$ are circuits of cardinality $p$ with leading 
terms of $\partial(C_r)$ and $\partial(C_s)$ greater than $M$, since 
we replaced the variable $y_{j_r}$ or $y_{j_s}$ by $y_u$, with $u>j_s>j_r$. 
The choice of $M$ and $C$ now implies that 
$\partial(C_r)$ and $\partial(C_s)$ are both in $I_2$, a contradiction.
\end{proof}

\section{Combinatorial Syzygies}\label{sec:three}
Example~\ref{Yuz} shows that the module of first syzygies on $I_2$ is
not determined by combinatorial data. In this section we study
{\em linear} first syzygies. First, we examine the syzygies which
arise from an $X\in L_2(\A)$ with $\mu(X)\geq 3$.
If $\mu(X)=d-1$, then $d$ hyperplanes $H_1,\ldots, H_d$ pass thru $X$. To simplify,
we localize to the rank two flat, so that $\A$ consists of $d$ points in ${\mathbb P}^1$.
\begin{thm}\label{syz2} Suppose $X\in L_2(\A)$ has $\mu(X)=d-1 \geq 3$ and 
let $I \subseteq R = \mathbb K[y_1,\ldots,y_d] \subseteq \mathbb K[y_1,\ldots,y_n]$ be the 
subideal of $I_2$ corresponding to $\A_X$. The ideal $I$ has an
Eagon-Northcott resolution \vskip .03in
\begin{small}
\[
\cdots \rightarrow S_2(R^2)^* \otimes \Lambda^4 R(-1)^{d-1} \rightarrow (R^2)^* \otimes \Lambda^3 R(-1)^{d-1} \rightarrow \Lambda^2 R(-1)^{d-1} \rightarrow \Lambda^2 R^2 \rightarrow R/I \rightarrow 0.
\]
\end{small}
\vskip .03in
In particular, the only nonzero betti numbers are
\[
\dim_{\mathbb K}Tor_i(R/I,\mathbb K)_{i+1} = i\cdot{ d-1 \choose i+1}.
\]
\end{thm}
\begin{proof} 
Let $X \in L_2(\A)$ with $\mu(X)=d-1$. After a change of coordinates,
$X=V(x_1,x_2)$, and $X \in H$ iff $l_H \in \langle x_1,x_2 \rangle$. Localization
is exact, so without loss of generality we may assume that $\A$ consists of 
$d$ points in ${\mathbb P}^1$. By Proposition~\ref{ss}, $I$ is quadratic, and
by Proposition~\ref{Prime}, $I$ has codimension $d-2$, so $V(I)$ is
an irreducible curve in ${\mathbb P}^{d-1}$. Since the irrelevant maximal ideal is not an associated prime, $I$ has depth at least one, and by 
Corollary~\ref{nondegen}, $X$ is not contained in any hyperplane, so 
$y_d$ is a nonzero divisor on $R/I$. This implies that 
$\deg X=\deg V(I +\langle y_d\rangle)$.
Since $\mathcal A$ has rank two, any set of three hyperplanes 
is dependent and thus every triple $\{H_i,H_j,H_k\}$ yields an element of $I$. 
Therefore
\[
\langle I,y_d \rangle= \langle J,y_d \rangle,
\mbox{ where }J=\langle \{y_iy_j\}_{1\leq i<j\leq d-1}\rangle.
\]
It follows that the primary decomposition of $\langle I,y_d\rangle$ is 
\[
\langle I,y_d\rangle=\bigcap_{1\leq i_1<\cdots<i_{d-2}\leq d-1}\langle y_{i_1},\ldots,y_{i_{d-2}},y_d \rangle,
\]
so $\deg V(\langle I,y_d\rangle)=d-1$. 
Any smooth, irreducible nondegenerate 
curve of degree $d-1$ in ${\mathbb P}^{d-1}$ is 
a rational normal curve, and has an Eagon-Northcott resolution 
\cite{eis}.
So we need only show that $V(I)$ is smooth. Our main theorem implies
$\A$ is 2-formal, and the proof of that result shows that 
\[
\dim T_p(V(I)) = 1
\]
for all $p \in V(I)$. Hence $V(I)$ is rational normal curve
of degree $d-1$.
\end{proof}
\subsection{Graphic arrangements}
For a graphic arrangement, all weights of dependencies are $\pm 1$, so the 
Orlik-Terao ideal $I_G$ has a presentation that is essentially 
identical to that of the Orlik-Solomon ideal obtained in \cite{SS02}. 
The proof of the next lemma is straightforward.
\begin{lem}\label{graphGens}
$I_G$ is minimally generated by the chordless cycles of $G$.
\end{lem}

\begin{lem}\label{linSyzG}
Every $K_4$ subgraph yields two minimal linear first syzygies on $I_G$.
\end{lem}
\begin{proof} Let $K_4$ be the complete graph on $\{1,2,3,4\}$. There are four relations:
\[
y_{12}+y_{23}-y_{13}, \mbox{ }
y_{12}+y_{24}-y_{14}, \mbox{ }
y_{13}+y_{34}-y_{14}, \mbox{ }
y_{23}+y_{34}-y_{24}.
\]
Let $y_1=y_{12}, y_2=y_{13}, y_3=y_{14}, y_4=y_{23}, y_5=y_{24}, y_6=y_{34}$. $I$ is generated by:
\begin{eqnarray}
f_4&=& y_2y_4 + y_1y_2-y_1y_4 \nonumber \\
f_3&=& y_3y_5+y_1y_3-y_1y_5 \nonumber \\
f_2&=& y_3y_6+y_2y_3-y_2y_6 \nonumber \\
f_1&=& y_5y_6+y_4y_5-y_4y_6 \nonumber
\end{eqnarray}
A direct calculation yields the pair of (independent) linear syzygies:
\begin{eqnarray}
(y_1-y_2)f_1-(y_1+y_5)f_2+(y_2+y_6)f_3-(y_6-y_5)f_4=0 \nonumber \\
(y_2-y_3)f_1-(y_4-y_5)f_2+(y_4-y_2)f_3-(y_5-y_3)f_4=0 \nonumber 
\end{eqnarray}
A Hilbert function computation as in Corollary~\ref{CIcrit} 
concludes the proof.
\end{proof}
\begin{thm}\label{linStrand}
If $\kappa_i$ is the number of induced subgraphs of type $K_{i+1} \subseteq G$,
then 
\[
Tor_i(R/I_G)_{i+1} = \begin{cases}
\kappa_2 & i = 1\\
2\kappa_3 & i=2\\
0 & i \ge 3.
\end{cases}
\]
\end{thm}
\begin{proof}
Combine a Hilbert function argument as in Corollary~\ref{CIcrit} 
with suitable modifications to the proofs of Corollary 6.6 and 
Lemma 6.9 of \cite{SS02}.
\end{proof}
\subsection{The spaces $R_k(\mathcal A)$}
In \cite{bt}, Brandt-Terao introduce higher relation spaces. 
\begin{defn}\label{defn:highRel}
For $X\in L_2(\mathcal A)$, let $F(\mathcal A_X)$ be the 
subspace of $F(\mathcal A)$ generated by the relations associated 
to circuits of length 3 $\{H_i,H_j,H_k\}$, with $X\subset H_i, H_j,H_k$. The inclusion map 
$F(\mathcal A_X)\hookrightarrow F(\mathcal A)$ gives a map 
\[
\pi: \bigoplus_{X\in L_2(\mathcal A)}F(\mathcal A_X)
\longrightarrow F(\mathcal A).
\]
The \textit{first relation space} $\mathcal R_3(\mathcal A)=\ker \pi.$
\end{defn}
The space $\mathcal R_3(\mathcal A)$ captures the dependencies among 
the circuits of length 3 in $\A$, and $\A$ is 2-formal iff $\pi$ is surjective.
It is clear that $\dim F(\A_X)=\mu(X)-1,$ 
where $\mu$ is the M\"obius function.
\begin{prop}\label{syz4} Let $X_1,\ldots,X_s\in L_2(\A)$ and let 
$r_1\in F(\A_{X_1}),\ldots,r_s\in F(\A_{X_s})$ be nonzero relations. 
If $$L_1\partial(r_1)+\cdots+L_s\partial(r_s)=0, L_i\neq 0$$ is a 
linear syzygy, then there exist $a_i\in\mathbb K$ 
with $a_1r_1+\cdots+a_sr_s=0.$
\end{prop}
\begin{proof} Suppose 
$$L_1\partial(r_1)+\cdots+L_s\partial(r_s)=0, L_i\neq 0$$ is a linear syzygy, with 
$L_i=a_1^iy_1+\cdots+a_d^iy_d$. Define $supp(L_i)$ to be the set of 
indices $j$ for which $a_j^i\neq 0$, and suppose 
$supp(L_i)\cap supp(r_i)\neq \emptyset$ for some $i$. 
In the expression of $L_i\partial(r_i)$ there exists a nonzero monomial 
of the form $y_uy_v^2$ with $H_u\cap H_v=X_i$. Since in the syzygy this 
monomial must be cancelled, there exists $k\neq i$ such 
that $L_k\partial(r_k)$ contains a term $y_uy_v^2$. Since $y_v^2$ cannot 
occur in $\partial(r_k)$, we see that $X_k=H_u\cap H_v$, contradicting 
$X_k\neq X_i$.

\noindent Let $\Lambda_i = supp(r_i)$. Substituting 
$(\frac{1}{y_1},\ldots,\frac{1}{y_d})$ in the syzygy yields:
\[
\frac{f_1}{y_{[d]\setminus\Lambda_1}}\frac{r_1}{y_{\Lambda_1}}+\cdots+ \frac{f_s}{y_{[d]\setminus\Lambda_s}}\frac{r_s}{y_{\Lambda_s}}=0,
\]
where $\frac{f_i}{y_{[d]\setminus\Lambda_i}}=L_i(\frac{1}{y_1},\ldots,\frac{1}{y_d})$ 
and $f_i\in\mathbb K[y_1,\ldots, y_d]$ are square-free non-zero polynomials, possibly divisible
 by a 
square-free monomial. Hence $f_1r_1+\cdots+f_sr_s=0$ and 
$f_1r_1 \in\langle r_2,\ldots,r_s\rangle$.
If $r_1\notin \langle r_2,\ldots,r_s\rangle$, then $f_1\in \langle r_2,\ldots,r_s\rangle$ 
and $$f_1=P_2r_2+\cdots+P_sr_s.$$ Evaluating this expression 
at $(\frac{1}{y_1},\ldots,\frac{1}{y_d})$ and clearing the denominators 
shows there exists a monomial $m$ such that 
$mL_1\in\langle \partial(r_2),\ldots, \partial(r_s)\rangle \subseteq I.$ 
By Proposition ~\ref{Prime} and Corollary ~\ref{nondegen}, $I$ is 
nondegenerate prime ideal, so this is impossible, and $r_1\in \langle r_2,\ldots,r_s\rangle$.
\end{proof}
\noindent{\bf Acknowledgement} Computations were performed using Macaulay2,
by Grayson and Stillman, and available at: {\tt http://www.math.uiuc.edu/Macaulay2/}
We also thank Dinh Van Le and Fatemeh Mohammadi for pointing out the missing ideal
quotient in Theorem 2.5.
\bibliographystyle{amsalpha}

\end{document}